\documentclass{amsart}    
\usepackage{amsfonts}
\usepackage{amssymb}
\usepackage{latexsym}
\usepackage{graphicx}

\newcommand{\nullform}{\mathfrak B_{\theta_{12}}}
\newcommand{\truenullform}{\mathfrak B_{\theta_{12} \ll 1}}

\newcommand{\Nmin}{N_{\mathrm{min}}}

\newcommand{\Lmin}{L_{\mathrm{min}}}
\newcommand{\Lmed}{L_{\mathrm{med}}}
\newcommand{\Lmax}{L_{\mathrm{max}}}

\newcommand{\hypwt}{\mathfrak h}

\newcommand{\abs}[1]{\left\vert #1 \right\vert}
\newcommand{\bigabs}[1]{\bigl\vert #1 \bigr\vert}

\newcommand{\norm}[1]{\left\Vert #1 \right\Vert}

\newcommand{\bignorm}[1]{\bigl\Vert #1 \bigr\Vert}
\newcommand{\Bignorm}[1]{\Bigl\Vert #1 \Bigr\Vert}

\newcommand{\N}{\mathbb{N}}
\newcommand{\Proj}{\mathbf{P}}

\newcommand{\R}{\mathbb{R}}

\DeclareMathOperator{\supp}{supp}

\newtheorem{theorem}{Theorem}[section]
\newtheorem{lemma}{Lemma}[section]

\theoremstyle{definition}

\theoremstyle{remark}

\title[Bilinear estimates]{Bilinear Fourier restriction estimates related to the 2d wave equation}

\author{Sigmund Selberg}
\address{Department of Mathematical Sciences\\ Norwegian University of Science and Technology\\ Alfred Getz' vei 1\\ N-7491 Trondheim\\ Norway}
\email{sselberg@math.ntnu.no}
\urladdr{www.math.ntnu.no/~sselberg}
\thanks{Supported by Research Council of Norway, grant no.~185359}

\begin{document}

\begin{abstract}
We study bilinear $L^2$ Fourier restriction estimates which are related to the 2d wave equation in the sense that we restrict to subsets of thickened null cones. In an earlier paper we studied the corresponding 3d problem, obtaining several refinements of the Klainerman-Machedon type estimates. The latter are bilinear generalizations of the $L^4$ estimate of Strichartz for the 3d wave equation. In 2d there is no $L^4$ estimate for solutions of the wave equation, but as we show here, one can nevertheless obtain $L^2$ bilinear estimates for thickened null cones, which can be viewed as analogues of the 3d Klainerman-Machedon type estimates. We then prove a number of refinements of these estimates, analogous to those we obtained earlier in 3d. The main application we have in mind is the Maxwell-Dirac system.
\end{abstract}

\maketitle


\section{Introduction}\label{A}

We are interested in bilinear $L^2$ Fourier restriction estimates on $\R^{1+2}$ of the form
\begin{equation}\label{Bilinear}
  \norm{\Proj_{A_0}(\Proj_{A_1} u_1 \cdot \Proj_{A_2} u_2)}
  \le C \norm{\Proj_{A_1} u_1} \norm{\Proj_{A_2} u_2}
  \qquad \forall u_1,u_2 \in \mathcal S(\R^{1+2}),
\end{equation}
for given measurable sets $A_0,A_1,A_2 \subset \R^{1+2}$. Here $\norm{\cdot}$ is the norm on $L^2(\R^{1+2})$ and the multiplier $\Proj_{A}$ is defined by $\widehat{\Proj_A u} = \chi_A \widehat u$, where $\chi_A$ is the characteristic function of $A$ and
$$\widehat u(\tau,\xi) = \mathcal F u(\tau,\xi) = \iint e^{-i(t\tau + x\cdot\xi)} u(t,x) \, dt \, dx$$ is the Fourier transform on $\R^{1+2}$. More specifically we are interested in estimates related to the 2d wave equation
$$
  \square u = 0 \qquad (\square = -\partial_t^2 + \Delta_x,\; t \in \R, \; x \in \R^2),
$$
in the sense that $A_1$ and $A_2$ (and possible also $A_0$) are thickened subsets of the null cone $K = \{ (\tau,\xi) \in \R^{1+2} \colon \abs{\tau} = \abs{\xi} \}$ (the characteristic set of $\square$). We introduce the following notation for thickened upper ($\tau \ge 0$) and lower ($\tau \le 0$) cones, truncated in the spatial frequency $\xi$ by balls, annuli and angular sectors:
\begin{align*}
  K^{\pm}_{N,L} &= \left\{ (\tau,\xi) \in \R^{1+2} \colon \abs{\xi} \lesssim N,\;\; \tau=\pm\abs{\xi}+O(L) \right\},
  \\
  \dot K^{\pm}_{N,L} &= \left\{ (\tau,\xi) \in \R^{1+2} \colon \abs{\xi} \sim N,\;\; \tau=\pm\abs{\xi}+O(L) \right\},
  \\
  \dot K^{\pm}_{N,L,\gamma}(\omega)
  &= \left\{ (\tau,\xi) \in \dot K^{\pm}_{N,L} \colon \theta(\pm\xi,\omega) \le \gamma \right\},
\end{align*}
where $N, L, \gamma > 0$, $\omega \in \mathbb S^1$ (the unit circle) and $\theta(a,b)$ denotes the angle between nonzero $a,b \in \R^2$. We use the shorthand $X \lesssim Y$ for $X \le CY$, where $C \gg 1$ is some absolute constant; $X=O(R)$ is short for $\abs{X} \lesssim R$; $X \sim Y$ means $X \lesssim Y \lesssim X$; $X \ll Y$ stands for $X \le C^{-1} Y$.

In the paper \cite{Selberg:2008c} we studied the corresponding problem on $\R^{1+3}$, obtaining several refinements of the well-known Klainerman-Machedon type estimates (see \cite{Klainerman:1993, Klainerman:1996, Foschi:2000}), which are bilinear generalizations of the $L^4$ estimate of Strichartz \cite{Strichartz:1977} for the 3d wave equation. All these estimates have an equivalent statement as Fourier restriction estimates for thickened cones. For example, Strichartz' estimate can be formulated as (in this paragraph all norms are taken on $\R^{1+3}$ instead of $\R^{1+2}$)
$$
  \norm{\Proj_{K^{\pm}_{N,L}} u}_{L^4} \lesssim N^{1/2} L^{1/2} \norm{u}  \qquad \forall u \in \mathcal S(\R^{1+3}),
$$
and the Klainerman-Machedon type estimates generalize this to
$$
  \norm{\Proj_{\abs{\xi} \lesssim N_0} \left( \Proj_{K^{\pm_1}_{N_1,L_1}} u_1 \cdot \Proj_{K^{\pm_2}_{N_2,L_2}} u_2 \right)}
  \lesssim \left( \Nmin^{012} \Nmin^{12} L_1 L_2 \right)^{1/2}
  \norm{u_1}\norm{u_2}
$$
for all $u_1, u_2 \in \mathcal S(\R^{1+3})$. Here $\Nmin^{012}$ stands for the minimum of $N_j$ for $j=0,1,2$, and similarly for $\Nmin^{12}$. The corresponding estimates for solutions of the wave equation appear when one sends the $L$'s to zero, the key point being that all the $L$'s in the right hand sides are to the power $1/2$.

Our aim here is to prove analogous results in 2d. At first glance this may seem not to make sense, since there is no $L^4$ estimate for solutions of the 2d wave equation. However, one can still prove estimates for thickened cones, although they will not correspond to estimates for solutions of the wave equation. For example, we show that
$$
  \norm{\Proj_{K^{\pm}_{N,L}} u}_{L^4} \lesssim N^{3/8} L^{3/8} \norm{u}  \qquad \forall u \in \mathcal S(\R^{1+2}),
$$
and more generally,
\begin{multline}\label{2dMain}
  \norm{\Proj_{\abs{\xi} \lesssim N_0} \left( \Proj_{K^{\pm_1}_{N_1,L_1}} u_1 \cdot \Proj_{K^{\pm_2}_{N_2,L_2}} u_2 \right)}
  \\
  \lesssim \left[ \Nmin^{012} (\Nmin^{12})^{1/2} \Lmin^{12} (\Lmax^{12})^{1/2} \right]^{1/2}
  \norm{u_1}\norm{u_2}
\end{multline}
for all $u_1, u_2 \in \mathcal S(\R^{1+3})$. Since the mean of the powers of the $L$'s in the right hand side is $3/8$, which is smaller than $1/2$, we do not obtain any estimate for solutions of the 2d wave equation by letting the $L$'s tend to zero. This is not a problem, however, since the application we have in mind is regularity theory for nonlinear wave equations (specifically the Maxwell-Dirac system), where one needs estimates for the nonlinear terms in $X^{s,b}$ spaces adapted to the null cone, the building blocks of which are estimates like \eqref{2dMain}.

The wave-adapted $X^{s,b}$ spaces arise naturally if one considers the problem of applying a dispersive estimate for free waves to a nonlinear wave equation, the idea being to express the iterates of the nonlinear problem as a kind of superposition of free waves by foliating the Fourier space using translates of cones. Because of this, the traditional way of deriving estimates for the nonlinear terms in $X^{s,b}$ spaces has been to start from a dispersive estimate for free waves and then derive the $X^{s,b}$ estimates from it by the foliation procedure (see e.g.~the survey \cite{Selberg:2002b}). But in this way one would never obtain \eqref{2dMain}, of course, since it simply does not correspond to any free wave estimate. To prove \eqref{2dMain} (and its various refinements stated in the next section) we use a direct approach based on the Cauchy-Schwarz inequality to reduce to estimating the $1+2$-dimensional volume of the intersection of various subsets of thickened cones, often combined with additional angular decompositions and various orthogonality arguments. This approach to proving multilinear $X^{s,b}$ estimates in general was studied in detail in \cite{Tao:2001}, and we shall we use many of the techniques developed there.

It should be remarked that although there is no $L^4$ estimate for solutions of the 2d wave equation, one can nevertheless obtain $L^2$ bilinear estimates for such solutions by replacing the product $u_1u_2$ with a null form; see \cite{Klainerman:1996}. We prove some refined null form estimates in the next section.

We shall need the following two elementary lemmas, more general versions of which can be found in \cite{Tao:2001}. Proofs of the lemmas as stated here can be found in \cite{Selberg:2008c}. We denote by $\abs{E}$ the volume of a set $E \subset \R^{1+2}$ (resp.~the area of a set $E \subset \R^2$ or the length of an interval $E \subset \R$). We shall write $X=(\tau,\xi)$ for the space-time Fourier variable.

\begin{lemma}\label{B:Lemma1} The estimate \eqref{Bilinear} holds with $C^2$ equal to an absolute constant times
$$
  \min\left(\sup_{X \in A_0} \abs{A_1 \cap (X-A_2)},\sup_{X \in A_2} \abs{A_0 \cap (X+A_1)},\sup_{X \in A_1} \abs{A_0 \cap (X+A_2)}\right),
$$
provided that this quantity is finite.
\end{lemma}

In fact, under certain hypotheses on $A_0$ one can take the intersection of translates of all three sets at once. We say that $A_0$ is an \emph{approximate tiling set} if, for some lattice $E \subset \R^{1+2}$, the family of translates $X+A_0$ with $X \in E$ is a cover of $\R^{1+2}$ with \emph{$O(1)$ overlap}, in the sense that there exists $M \in \N$ such that for any $X \in E$, there are at most $M$ vectors $X' \in E$ such that $X'+A_0$ and $X+A_0$ intersect. But since $E$ is a lattice, this is equivalent to saying that $\left\{ X \in E : (X + A_0) \cap A_0 \neq \emptyset \right\}$ has cardinality at most $M$. Further, defining the \emph{doubling} $A_0^* = A_0 + A_0$, we say that $A_0$ has the \emph{doubling property} if the cover $\left\{ X + A^* \right\}_{X \in E}$ also has $O(1)$ overlap.

\begin{lemma}\label{B:Lemma2} Suppose $A_0$ is an approximate tiling set with the doubling property. Then \eqref{Bilinear} holds with
$$
  C^2
  \sim
  \sup_{X \in A_0, \; X' \in E} \bigabs{A_1 \cap (X-A_2) \cap (X'+A_0)},
$$
provided that this quantity is finite.
\end{lemma}

Before presenting our main results we introduce some notation and terminology. Note the convolution formula
\begin{equation}\label{A:130}
  \widehat{u_1 u_2}(X_0)
  \simeq
  \int \widehat{u_1}(X_1)\widehat{u_2}(X_2)  
  \, \delta(X_0-X_1-X_2) \, dX_1 \, dX_2,
\end{equation}
where $X_j = (\tau_j,\xi_j) \in \R^{1+2}$, $j=0,1,2$. Here $X_0 = X_1 + X_2$, so in particular $\xi_0=\xi_1+\xi_2$, hence $\abs{\xi_j} \le \abs{\xi_k} + \abs{\xi_l}$ for all permutations $(j,k,l)$ of $(0,1,2)$. Therefore, one of the following must hold:
\begin{subequations}\label{A:140}
\begin{alignat}{2}
  \label{A:140a}
  \abs{\xi_0} &\ll \abs{\xi_1} \sim \abs{\xi_2}& \qquad &(\text{``low output''}),
  \\
  \label{A:140b}
  \abs{\xi_0} &\sim \max(\abs{\xi_1},\abs{\xi_2})& \qquad &(\text{``high output''}).
\end{alignat}
\end{subequations}
The integration in \eqref{A:130} may be restricted to $\xi_1,\xi_2 \neq 0$, hence the angle $\theta_{12} = \theta(\pm_1\xi_1,\pm_2\xi_2)$ is well-defined. Given signs $\pm_j$, we define also $\hypwt_j =  -\tau_j\pm_j\abs{\xi_j}$, which we call the \emph{hyperbolic weights}, whereas the $\abs{\xi_j}$ are called \emph{elliptic weights}.

\section{Main results}

The estimate \eqref{2dMain} is contained (let $L_0 \to \infty$) in the following result.

\begin{theorem}\label{A:Thm1} The estimate
\begin{equation}\label{A:100}
  \norm{\Proj_{K^{\pm_0}_{N_0,L_0}} \left( \Proj_{K^{\pm_1}_{N_1,L_1}} u_1 \cdot \Proj_{K^{\pm_2}_{N_2,L_2}} u_2 \right)}
  \le C
  \norm{u_1}
  \norm{u_2}
\end{equation}
holds with
\begin{align}
\label{A:110}
  C
  &\sim \bigl( \Nmin^{012}\Lmin^{12} \bigr)^{1/2} \bigl( \Nmin^{12}\Lmax^{12} \bigr)^{1/4},
  \\
  \label{A:112}
  C
  &\sim \bigl( \Nmin^{012}\Lmin^{0j} \bigr)^{1/2} \bigl( \Nmin^{0j}\Lmax^{0j} \bigr)^{1/4} \qquad (j=1,2),
  \\
  \label{A:116}
  C
  &\sim \bigl( (\Nmin^{012})^2 \Lmin^{012} \bigr)^{1/2},
\end{align}
regardless of the choice of signs $\pm_j$.
\end{theorem}

Here \eqref{A:116} is immediate from Lemma~\ref{B:Lemma1}, and \eqref{A:112} follows from \eqref{A:110} by symmetry, so the key estimate that needs to be proved is \eqref{A:110} [or equivalently \eqref{2dMain}]. For its proof, given in section~\ref{C}, we rely in part on the following elementary estimate for the area of intersection of two thickened circles. Writing $\mathbb S^1_\delta(r) = \left\{ \xi \in \R^2 \colon \abs{\xi} = r + O(\delta) \right\}$, we have (this is proved in section \ref{KK}):

\begin{lemma}\label{A:Lemma1}
Let $0 < \delta \ll r$ and $0 < \Delta \ll R$. Then for any $\xi_0 \in \R^2 \setminus \{0\}$,
$$
  \abs{\mathbb S^1_\delta(r) \cap [ \xi_0 + \mathbb S^1_\Delta(R)]}
  \lesssim \left[\frac{rR\delta\Delta }{\abs{\xi_0}} \min(\delta,\Delta) \right]^{1/2}.
$$
\end{lemma}

Due to the factor $(\Nmin^{012})^{1/2}$ in \eqref{A:110}, it suffices (see section~\ref{C}) to prove the estimate with the  balls replaced by annuli, i.e.~with $K^{\pm_j}_{N_j,L_j}$ replaced by $\dot K^{\pm_j}_{N_j,L_j}$ for $j=0,1,2$. We now present various modifications and refinements of the estimates in Theorem \ref{A:Thm1}, and keep using annuli instead of balls (at certain points in the proofs, this is essential), assuming throughout that $u_1, u_2 \in L^2(\R^{1+3})$ satisfy $\widehat{u_j} \subset \dot K^{\pm_j}_{N_j,L_j}$.

The common thread in the following results is the question: To what extent do additional Fourier restrictions lead to improvements in the estimates? All the results that follow are analogues of estimates in 3d proved in \cite{Selberg:2008c}.

\subsection{An anisotropic estimate}

Note that \eqref{A:110} implies, assuming $\widehat{u_j} \subset \dot K^{\pm_j}_{N_j,L_j}$,
$$
  \norm{u_1u_2} \lesssim \left[(\Nmin^{12})^{3/2}(L_1L_2)^{3/4}\right]^{1/2}
  \norm{u_1}\norm{u_2}.
$$
Restricting the spatial output frequency $\xi_0$ to a ball $B \subset \R^2$ of radius $r \ll \Nmin^{12}$ and arbitrary center, the estimate improves to
$$
  \norm{\Proj_{\R \times B}(u_1u_2)} \lesssim \left[r(\Nmin^{12})^{1/2}(L_1L_2)^{3/4}\right]^{1/2}
  \norm{u_1}\norm{u_2}.
$$
This follows immediately from the following much more powerful anisotropic estimate.

\begin{theorem}\label{Z:Thm1}
Let $\omega \in \mathbb S^1$, and let $I \subset \R$ be a compact interval. Assume $\widehat{u_1}$ is supported outside an angular neighborhood of the orthogonal complement $\omega^\perp$ of $\omega$:
$$
  \supp \widehat{u_1} \subset
  \left\{ (\tau,\xi) \colon
  \theta(\xi,\omega^\perp) \ge \alpha \right\}
  \qquad \text{for some $\;0 < \alpha \ll 1$}.
$$
Assuming also $\widehat{u_j} \subset \dot K^{\pm_j}_{N_j,L_j}$ for $j=1,2$, we then have
$$
  \norm{\Proj_{\xi_0 \cdot \omega \in I}
  (u_1 u_2)}
  \lesssim \left[ \frac{\abs{I}(\Nmin^{12})^{1/2}(L_1L_2)^{3/4}}{\alpha} \right]^{1/2}
  \norm{u_1}
  \norm{u_2},
$$
where $\abs{I}$ is the length of $I$.
\end{theorem}

The proof can be found in section \ref{E}. We remark that, since the position of $I$ is arbitrary, $\Proj_{\xi \cdot \omega \in I}$ can equivalently be placed in front of either $u_1$ or $u_2$, as can be seen by a standard tiling argument. 

\subsection{Null form estimates}\label{N}

Next we discuss estimates where the product $u_1u_2$ is replaced by the bilinear null form $\nullform(u_1,u_2)$, defined on the Fourier transform side by inserting the angle $\theta_{12} = \theta(\pm_1\xi_1,\pm_2\xi_2)$ in the integral in \eqref{A:130}. This angle improves matters when we are close to the null interaction in \eqref{A:130}, i.e.~when $X_1$, $X_2$ and $X_0 = X_1+X_2$ are all approximately on the null cone. This improvement is quantified by the following rather standard lemma (a proof of the version formulated here can be found in \cite{Selberg:2008c}): 

\begin{lemma}\label{I:Lemma1} Assume that $X_0 = X_1 + X_2$, where $X_j = (\tau_j,\xi_j)$ with $\xi_j \neq 0$ for $j=0,1,2$, and set $\hypwt_j =  -\tau_j\pm_j\abs{\xi_j}$. Then for all choices of signs $\pm_j$,
$$
  \max\left( \abs{\hypwt_0}, \abs{\hypwt_1}, \abs{\hypwt_2} \right)
  \gtrsim
  \min\left(\abs{\xi_1},\abs{\xi_2}\right)\theta_{12}^2.
$$
Moreover, if $\pm_0$ is chosen so that
$
  \abs{\hypwt_0} = \bigabs{\abs{\tau_0} - \abs{\xi_0}},
$
and if
$
  \abs{\hypwt_1}, \abs{\hypwt_2} \ll \abs{\hypwt_0},
$
then
$$
  \abs{\hypwt_0} \sim
  \left\{
  \begin{alignedat}{2}
  &\min\left(\abs{\xi_1},\abs{\xi_2}\right)\theta_{12}^2& \quad &\text{if $\pm_1=\pm_2$},
  \\
  &\frac{\abs{\xi_1}\abs{\xi_2}\theta_{12}^2}{\abs{\xi_0}}& \quad &\text{if $\pm_1\neq\pm_2$}.
  \end{alignedat}
  \right.
$$
\end{lemma}

For example, combining the first part of the lemma with Theorem \ref{A:Thm1}, we immediately obtain the following null form estimate, assuming $\widehat{u_j} \subset \dot K^{\pm_j}_{N_j,L_j}$:
$$
  \Bignorm{ \Proj_{K^{\pm_0}_{N_0,L_0}}\!\!
  \nullform(u_1,u_2)}
  \lesssim \left[ N_0 \Lmin^{012} (\Lmed^{012}\Lmax^{012})^{1/2} \right]^{1/2}
  \norm{u_1}\norm{u_2}.
$$

The key point in the next result is that we are able to exploit concentration of the Fourier supports along null rays, due to the null form; for a standard product such concentrations do not give any improvement, since in the worst case the thickened cones already intersect along a null ray (approximately, assuming $L_1, L_2$ small relative to $N_1, N_2$).

\begin{theorem}\label{N:Thm2} Given $r > 0$ and $\omega \in \mathbb S^1$, let $T_r(\omega) \subset \R^2$ be the strip of width $r$ centered on $\R\omega$. Then, assuming $u_1,u_2$ satisfy $\widehat{u_j} \subset \dot K^{\pm_j}_{N_j,L_j}$, 
$$
  \norm{\nullform(\Proj_{\R \times T_r(\omega)} u_1,u_2)}
  \lesssim \left( r L_1 L_2 \right)^{1/2}
  \norm{u_1}\norm{u_2}.
$$
\end{theorem}

The proof is given in section \ref{E}.

\subsection{Concentration/nonconcentration null form estimate}\label{Z}

It is natural to ask whether there is any improvement in Theorem \ref{N:Thm2} if we restrict the output $\xi_0$ to a ball $B \subset \R^2$ of radius $\delta$ and arbitrary center. Inserting $\Proj_{\R \times B}$ in front of the null form we obviously get an improvement if $\delta \lesssim r$, since then we can apply Theorem \ref{N:Thm3}. So let us assume $\delta \gg r$. Then Fourier restriction to $B$ will have no effect in directions perpendicular to $\omega$, so we may as well replace $\Proj_{\R \times B}$ by $\Proj_{\xi_0 \cdot \omega \in I_0}$, where $I_0 \subset \R$ is a compact interval of length $\abs{I_0} = \delta$. In this situation we have the following result, where we limit attention to interactions with $\theta_{12} \ll 1$; we denote this modified null form by $\truenullform$.
 
\begin{theorem}\label{N:Thm4}
Let $r > 0$, $\omega \in \mathbb S^1$ and $I_0 \subset \R$ a compact interval. Assume that $u_1,u_2$ satisfy $\widehat{u_j} \subset \dot K^{\pm_j}_{N_j,L_j}$, and that $r \ll \Nmin^{12}$. Then
$$
  \norm{\Proj_{\xi_0 \cdot \omega \in I_0} \truenullform(\Proj_{\R \times T_r(\omega)} u_1,u_2)}
  \lesssim \left( r L_1 L_2 \right)^{1/2} 
  \left( \sup_{I_1} \norm{\Proj_{\xi_1 \cdot \omega \in I_1} u_1} \right)
  \norm{u_2},
$$
where the supremum is over all translates $I_1$ of $I_0$.
\end{theorem}

The proof is given in section \ref{J}. The main point is that we get an improvement over Theorem \ref{N:Thm2} by concentrating the output in a strip $\xi_0 \cdot \omega \in I_0$, provided that the Fourier support of $u_1$ is not concentrated in a translated strip $\xi_1 \cdot \omega \in I_1$, but is more spread out. Because of this, we call it a concentration/nonconcentration null form estimate.

\subsection{Nonconcentration low output estimate}\label{L}

In the case $N_0 \ll N_1 \sim N_2$ (low output), Theorem \ref{A:Thm1} implies, assuming $\widehat{u_j} \subset \dot K^{\pm_j}_{N_j,L_j}$,
$$
  \Bignorm{\Proj_{K^{\pm_0}_{N_0,L_0}}\!\!(u_1u_2)}
  \lesssim \left[ N_0 N_1^{1/2} \Lmin^{012} (\Lmed^{012})^{1/2} \right]^{1/2} \norm{u_1} \norm{u_2}.
$$
In general this is optimal, as can be seen by testing it on functions which concentrate along a null ray in Fourier space, but one may hope to do better if the Fourier supports are less concentrated. To detect radial nonconcentration we introduce a modified $L^2$ norm as follows. Let
$$
  \Omega(\gamma) \subset \mathbb S^1 \qquad (0 < \gamma \le \pi )
$$
be a maximal $\gamma$-separated subset of the unit circle $\mathbb S^1$. Since the cardinality of $\Omega(\gamma)$ is comparable to $1/\gamma$, we see that for all $N,r > 0$, 
$$
  \norm{\Proj_{\abs{\xi} \sim N} u}
  \lesssim
  \norm{u}_{N,r}
  \equiv
  \left(\frac{N}{r}\right)^{1/2}
  \sup_{\omega \in \mathbb S^1}
  \norm{\Proj_{\abs{\xi} \sim N} \Proj_{T_r(\omega)} u},
$$
and the less radial concentration we have in the spatial Fourier support, the closer the two norms are to being comparable. In the extreme case of circular symmetry in $\xi$, we have $\norm{u} \sim \norm{\Proj_{\abs{\xi} \sim N} u}_{N,r}$. 

We then have the following result.

\begin{theorem}\label{L:Thm1} Assume $N_0 \ll N_1 \sim N_2$, and define $r = (N_0\Lmax^{012})^{1/2}$. Assume that $u_1,u_2 \in L^2(\R^{1+3})$ satisfy $\widehat{u_j} \subset \dot K^{\pm_j}_{N_j,L_j}$. Then
$$
  \Bignorm{\Proj_{K^{\pm_0}_{N_0,L_0}}\!\!(u_1u_2)}
  \lesssim \left( N_0 \Lmin^{012} (\Lmed^{012}\Lmax^{012})^{1/2} \right)^{1/2} \norm{u_1}
  \norm{u_2}_{N_2,r}.
$$
In other words,
$$
  \Bignorm{\Proj_{K^{\pm_0}_{N_0,L_0}}\!\!(u_1u_2)}
  \lesssim \left( N_0^{1/2} N_1 \Lmin^{012} (\Lmed^{012})^{1/2} \right)^{1/2} \norm{u_1}
  \sup_{\omega \in \mathbb S^1} \norm{\Proj_{\R \times T_{r}(\omega)}u_2}.
$$
\end{theorem}

We omit the proof, since it is just a repetition of the proof of the analogous 3d result, given in \cite{Selberg:2008c}, up to some obvious modifications. Specifically, one must use Theorem \ref{A:Thm1} instead of its 3d counterpart \cite[Theorem 1.1]{Selberg:2008c}, and one must replace \cite[Lemma 2.2]{Selberg:2008c} by its 2d analogue, which we now state. Writing
$$
  H_d(\omega) = \left\{ (\tau,\xi) \in \R^{1+2} \colon \abs{-\tau + \xi \cdot \omega} \le d \right\} \qquad (d > 0, \; \omega \in \mathbb S^1)
$$
for a thickening of the null hyperplane $-\tau + \xi \cdot \omega = 0$, we have:

\begin{lemma}\label{L:Lemma} Suppose $N,d,\gamma > 0$. Then
$$
  \sum_{\omega \in \Omega(\gamma)} \chi_{H_d(\omega)}(\tau,\xi)
  \lesssim 1 + \left( \frac{d}{N\gamma^2} \right)^{1/2}
  \qquad \text{for all $(\tau,\xi) \in \R^{1+2}$ with $\abs{\xi} \sim N$.}
$$
\end{lemma}

\begin{proof}
The left side equals $\# \left\{ \omega \in \Omega(\gamma) \colon \omega \in A \right\}$ where $A$ is the set of $\omega \in \mathbb S^1$ such that $\abs{-\tau+\xi\cdot\omega} \le d$, for given $\tau,\xi$ with $\abs{\xi} \sim N$. Without loss of generality assume $\xi = (\abs{\xi},0)$. Then
$$
  A \subset A' \equiv \left\{ \omega = (\omega^1,\omega^2) \in \mathbb S^1 \colon \omega^1 = \frac{\tau}{\abs{\xi}} + O\left(\frac{d}{N}\right) \right\}.
$$
Thus, $A'$ is the intersection of $\mathbb S^1$ and a thickened line with thickness comparable to $d/N$, so
$$
  \# \left\{ \omega \in \Omega(\gamma)\colon \omega \in A' \right\}
  \lesssim 1 + \frac{\text{length}(A')}{\gamma}.
$$
But $\text{length}(A') \lesssim (d/N)^{1/2}$, and the proof is complete.
\end{proof}

In preparation for the proofs of the above theorems we recall a few basic facts concerning decomposition of the spatial frequencies in a product into angular sectors.
For $\gamma \in (0,\pi]$ and $\omega \in \mathbb S^1$ we define
$$
  \Gamma_\gamma(\omega) = \left\{ \xi \in \R^2 \colon \theta(\xi,\omega) \le \gamma \right\}.
$$
Recall that $\Omega(\gamma)$ denotes a maximal $\gamma$-separated subset of $\mathbb S^1$. Thus,
\begin{equation}\label{A:200}
  1 \le \sum_{\omega \in \Omega(\gamma)} \chi_{\Gamma_\gamma(\omega)}(\xi)
  \le 5
  \qquad (\forall \xi \neq 0),
\end{equation}
where the left inequality holds by the maximality of $\Omega(\gamma)$, and the right inequality by the $\gamma$-separation, since the latter clearly implies
\begin{equation}\label{A:202}
  \#\left\{ \omega' \in \Omega(\gamma) : \theta(\omega',\omega) \le k\gamma \right\} \le 2k+1
\end{equation}
for any $k \in \N$ and $\omega \in \Omega(\gamma)$. We shall write
$$
  u_j^{\gamma,\omega} = \Proj_{\theta(\pm_j\xi,\omega) \le \gamma} u_j.
$$
Then by \eqref{A:200},
\begin{equation}\label{A:204}
  \norm{u_j}^2
  \sim
  \sum_{\omega \in \Omega(\gamma)}
  \norm{u_j^{\gamma,\omega}}^2,
\end{equation}
hence
\begin{equation}\label{A:208}
  \sum_{\genfrac{}{}{0pt}{1}{\omega_1,\omega_2 \in \Omega(\gamma)}
  {\theta(\omega_1,\omega_2) \lesssim \gamma}}
  \norm{u_1^{\gamma,\omega_1}} \norm{u_2^{\gamma,\omega_2}}
  \lesssim \norm{u_1}\norm{u_2},
\end{equation}
where we applied the Cauchy-Schwarz inequality and used \eqref{A:202} and \eqref{A:204}.

Next we note the following Whitney decomposition over angular variables.

\begin{lemma}\label{A:Lemma4} We have
$$
  1 \sim \sum_{\genfrac{}{}{0pt}{1}{0 < \gamma < 1}{\text{$\gamma$ \emph{dyadic}}}} 
  \sum_{\genfrac{}{}{0pt}{1}{\omega_1,\omega_2 \in \Omega(\gamma)}{3\gamma \le \theta(\omega_1,\omega_2) \le 12\gamma}}
  \chi_{\Gamma_\gamma(\omega_1)}(\xi_1) \chi_{\Gamma_\gamma(\omega_2)}(\xi_2),
$$
for all $\xi_1,\xi_2 \in \R^2 \setminus \{0\}$ with $\theta(\xi_1,\xi_2) > 0$.
\end{lemma}

We omit the straightforward proof. The condition $\theta(\omega_1,\omega_2) \ge 3\gamma$ implies that the minimum angle between vectors in $\Gamma_{\gamma}(\omega_1)$ and $\Gamma_{\gamma}(\omega_2)$ is greater than or equal to $\gamma$, so the sectors are well-separated. If separation is not needed, the following variation may be preferable (again, we skip the easy proof):

\begin{lemma}\label{A:Lemma5} For any $0 < \gamma < 1$ and $k \in \N$,
$$
  \chi_{\theta(\xi_1,\xi_2) \le k\gamma} \lesssim \sum_{\genfrac{}{}{0pt}{1}{\omega_1,\omega_2 \in \Omega(\gamma)}{\theta(\omega_1,\omega_2) \le (k+2)\gamma}}
  \chi_{\Gamma_\gamma(\omega_1)}(\xi_1) \chi_{\Gamma_\gamma(\omega_2)}(\xi_2),
$$
for all $\xi_1,\xi_2 \in \R^2 \setminus \{0\}$.
\end{lemma}

\section{Proof of the fundamental 2d estimate}\label{C}

Here we prove \eqref{2dMain}. We remark that for the weaker estimate where the $L$-weights in the right hand side are symmetrized [i.e.~we replace $\Lmin^{012}(\Lmax^{12})^{1/2}$ by $(L_1L_2)^{3/4}$], the proof would just be a repetition of the proof of the 3d counterpart \cite[Theorem 1.1]{Selberg:2008c}, but using Lemma \ref{A:Lemma1} instead of its 3d analogue. To get \eqref{2dMain} with the asymmetric $L$-weights, on the other hand, we need to work a little harder.

As remarked already, we may replace the $K^{\pm_j}_{N_j,L_j}$ by $\dot K^{\pm_j}_{N_j,L_j}$. Indeed, if we can prove the latter case, then decomposing the balls $\abs{\xi_j} \lesssim N_j$ into annuli $\abs{\xi_j} \sim M_j$ for dyadic $0 < M_j \le N_j$, it is easy to recover the former case by summing, using also the fact that the two largest of the $M_j$ are comparable. So we shall assume $\supp \widehat{u_j} \subset \dot K^{\pm_j}_{N_j,L_j}$ in \eqref{2dMain}. We may also assume $\Lmax^{12} \ll \Nmin^{012}$, as otherwise \eqref{A:116} is already better than \eqref{2dMain}. Recalling \eqref{A:140}, we split into the cases $N_0 \ll N_1 \sim N_2$ (LHH), $N_1 \lesssim N_0 \sim N_2$ (HLH) and $N_2 \lesssim N_0 \sim N_1$ (HHL), but by symmetry it suffices to consider LHH and HLH.

\subsection{The HLH case: $N_1 \lesssim N_0 \sim N_2$}

By Lemma \ref{B:Lemma1} we reduce to proving that, for any $(\tau_0,\xi_0) \in \R^{1+2}$ with $\abs{\xi_0} \sim N_0$, the set
$$
  E = \dot K^{\pm_1}_{N_1,L_1} \cap \bigl[ (\tau_0,\xi_0) - \dot K^{\pm_2}_{N_2,L_2} \bigr]
$$
verifies $\abs{E} \lesssim N_1^{3/2} \Lmin^{12} (\Lmax^{12})^{1/2}$. Denote the slices $\tau=\text{const}$ by $E(\tau)$, and set $J = \left\{ \tau \in \R \colon E(\tau) \neq \emptyset \right\}$, so by Fubini's theorem, $\abs{E} \le \abs{J} \cdot \sup_{\tau \in J} \abs{E(\tau)}$. But\begin{multline*}
  E \subset \big\{ (\tau,\xi) \in \R^{1+2} \colon \abs{\xi} \sim N_1, \;
  \tau=\pm_1\abs{\xi}+O(L_1),
  \\
  \abs{\xi_0-\xi} \sim N_2,\;
  \tau_0-\tau=\pm_2\abs{\xi_0-\xi}+O(L_2) \big\},
\end{multline*}
so recalling $\Lmax^{12} \ll N_1 \lesssim N_0 \sim N_2$ we see that $$E(\tau) \subset \mathbb S_{L_1}(\abs{\tau}) \cap [\xi_0 + \mathbb S_{L_2}(\abs{\tau_0-\tau})]$$ if $\abs{\tau} \sim N_1$, and otherwise $E(\tau)$ is empty. Hence $\abs{J} \lesssim N_1$, and by Lemma \ref{A:Lemma1},
$$
  \abs{E(\tau)}^2 \lesssim \frac{N_1N_2\left(\Lmin^{12}\right)^2 \Lmax^{12}}{N_0}
  \sim N_1\left(\Lmin^{12}\right)^2 \Lmax^{12},
$$
hence $\abs{E} \le \abs{J} \cdot \sup_{\tau \in J} \abs{E(\tau)} \lesssim N_1 \bigl[ N_1\left(\Lmin^{12}\right)^2 \Lmax^{12} \bigr]^{1/2}$ as needed.

\subsection{The LHH case: $N_0 \ll N_1 \sim N_2$}

We may assume $L_1 \le L_2$ by symmetry, and moreover $L_2 \ll N_0^2/N_1$, since otherwise \eqref{A:116} is better than \eqref{A:110}. So now
$$
  L_1 \le L_2 \ll \frac{N_0^2}{N_1} \ll N_0 \ll N_1 \sim N_2.
$$
The output $\xi_0$ being restricted to a ball $B_0 = \left\{ \xi \in \R^2 \colon \abs{\xi} \lesssim N_0 \right\}$, we may, by a standard tiling argument, restrict the spatial Fourier supports of $u_1$ and $u_2$ to translates $\xi_1'+B_0$ and $\xi_2'+B_0$, where $\xi_1',\xi_2' \in \R^2$ satisfy $\abs{\xi_1'+\xi_2'} \lesssim N_0$ and $\abs{\xi_1'} \sim \abs{\xi_2'} \sim N_1$, and then we have to prove
$$
  \norm{u_1u_2} \lesssim \left(N_0L_1\right)^{1/2} \left(N_1L_2\right)^{1/4} \norm{u_1} \norm{u_2}.
$$
Writing $\theta_{12} = \theta(\pm_1\xi_1,\pm_2\xi_2)$, we split into the cases $\theta_{12} \lesssim \gamma_0$ and $\theta_{12} \gg \gamma_0$, where
$$
  \gamma_0 = \left(\frac{L_2}{N_1}\right)^{1/2}
$$
is the critical angle which makes an angular sector of the $\widehat{u_2}$-support essentially flat (the condition being $N_1 \gamma_0^2 \sim L_2$). Assuming without loss of generality that $\widehat u_1, \widehat u_2 \ge 0$, we then have, by Lemmas \ref{A:Lemma4} and \ref{A:Lemma5},
\begin{align*}
  \norm{u_1u_2} &\lesssim \sum_{\omega_1,\omega_2 \in \Omega(\gamma_0)} \chi_{\theta(\omega_1,\omega_2) \lesssim \gamma_0} \norm{u_1^{\gamma_0,\omega_1}u_2^{\gamma_0,\omega_2}}
  \\
  & \qquad +
  \sum_{\gamma_0 < \gamma < 1} \sum_{\omega_1,\omega_2 \in \Omega(\gamma)} \chi_{3\gamma \le \theta(\omega_1,\omega_2) \le 12\gamma} \norm{u_1^{\gamma,\omega_1}u_2^{\gamma,\omega_2}}
  =: \Sigma_1 + \Sigma_2,
\end{align*}
where $\gamma$ is understood to be dyadic.

The volume of the support of $\mathcal F u_1^{\gamma_0,\omega_1}$ is $O(N_0 N_1\gamma_0 L_1) \sim O(N_0 \sqrt{N_1L_2} L_1)$, so Lemma \ref{B:Lemma1} gives
$$
  \norm{u_1^{\gamma_0,\omega_1}u_2^{\gamma_0,\omega_2}}
  \lesssim
  \left(N_0L_1\right)^{1/2} \left(N_1L_2\right)^{1/4}
  \norm{u_1^{\gamma_0,\omega_1}}
  \norm{u_2^{\gamma_0,\omega_2}},
$$
hence
$$
  \Sigma_1
  \lesssim
  \left(N_0L_1\right)^{1/2} \left(N_1L_2\right)^{1/4}
  \sum_{\omega_1,\omega_2 \in \Omega(\gamma_0)} \chi_{\theta(\omega_1,\omega_2) \lesssim \gamma_0}
  \norm{u_1^{\gamma_0,\omega_1}}
  \norm{u_2^{\gamma_0,\omega_2}},
$$
so applying \eqref{A:208} we get the desired estimate.

Now consider $\Sigma_2$. Since the sectors $\Gamma_{\gamma}(\omega_1)$ and $\Gamma_{\gamma}(\omega_2)$ are separated by an angle comparable to $\gamma$, we get from Lemma \ref{B:Lemma1} (see the next subsection for the details) that
\begin{equation}\label{C:200}
  \norm{u_1^{\gamma,\omega_1}u_2^{\gamma,\omega_2}}
  \lesssim
  \left( N_0 \frac{L_1L_2}{\gamma} \right)^{1/2}
  \norm{u_1^{\gamma,\omega_1}}
  \norm{u_2^{\gamma,\omega_2}},
\end{equation}
hence
\begin{align*}
  \Sigma_2
  &\lesssim
  \sum_{\gamma_0 < \gamma < 1}  \left( N_0 \frac{L_1L_2}{\gamma} \right)^{1/2}
  \sum_{\omega_1,\omega_2 \in \Omega(\gamma)}
  \norm{u_1^{\gamma,\omega_1}}
  \norm{u_2^{\gamma,\omega_2}}
  \\
  &\lesssim
  \left( N_0 \frac{L_1L_2}{\gamma_0} \right)^{1/2}
  \norm{u_1}
  \norm{u_2}
  \sim
  \left(N_0L_1\right)^{1/2} \left(N_1L_2\right)^{1/4}
  \norm{u_1}
  \norm{u_2},
\end{align*}
where we used \eqref{A:208} and $\sum_{\gamma_0 < \gamma < 1} \gamma^{-1/2} \sim \gamma_0^{-1/2}$.

\subsection{Proof of \eqref{C:200}} By Lemma \ref{B:Lemma1} we reduce to proving that $\abs{E} \lesssim N_0 L_1 L_2 / \gamma$ uniformly in $(\tau_0,\xi_0) \in \R^{1+2}$, where
\begin{equation}\label{K:12}
  E = \left\{ (\tau,\xi) \colon \xi \in R,\;\;
  \tau=\pm_1\abs{\xi}+O(L_1),\;\;
  \tau_0-\tau=\pm_2\abs{\xi_0-\xi}+O(L_2) \right\}
\end{equation}
and
$$
  R = \left\{ \xi \in \R^2 \colon \pm_1\xi \in \Gamma_{\gamma}(\omega_1), \;\;
  \pm_2(\xi_0-\xi) \in \Gamma_{\gamma}(\omega_2),\;\;
  \abs{\xi-\xi_1'} \lesssim N_0 \right\}.
$$
Integrating first in the $\tau$-direction, we get by Fubini's theorem:
\begin{equation}\label{K:16}
  \abs{E}
  \lesssim \Lmin^{12} \abs{ \left\{ \xi \in R \colon  \;\; f(\xi)=\tau_0+O(\Lmax^{12}) \right\} },
\end{equation}
where
\begin{equation}\label{K:20}
  f(\xi) = \pm_1\abs{\xi} \pm_2\abs{\xi-\xi_0}.
\end{equation}
Then $\nabla f(\xi) = e_1-e_2$, where
\begin{equation}\label{K:24}
  e_1 = \pm_1\frac{\xi}{\abs{\xi}},
  \qquad
  e_2 = \pm_2\frac{\xi_0-\xi}{\abs{\xi_0-\xi}}
\end{equation}
satisfy $\theta(e_1,\omega_1) \le \gamma$ and $\theta(e_2,\omega_2) \le \gamma$ for $\xi \in R$. Choosing coordinates $\xi=(\xi^1,\xi^2)$ so that $\frac{\omega_1-\omega_2}{\abs{\omega_1-\omega_2}} = (1,0)$ we have, for all $\xi \in R$,
\begin{align*}
  \partial_1 f(\xi)
  &= \frac{\cos\theta(e_1,\omega_1)+\cos\theta(e_2,\omega_2)-\cos\theta(e_1,\omega_2)-\cos\theta(e_2,\omega_1)}{\abs{\omega_1-\omega_2}}
  \\
  &\ge
  \frac{2\cos\gamma-2\cos2\gamma}{\abs{\omega_1-\omega_2}}
  \ge
  \frac{\gamma^2}{\abs{\omega_1-\omega_2}} \sim \gamma
\end{align*}
where we used the assumption $\theta(\omega_1,\omega_2) \ge 3\gamma$, which implies $\theta(e_1,\omega_2) \ge 2\gamma$, $\theta(e_2,\omega_1) \ge 2\gamma$ and $\abs{\omega_1 - \omega_2} \sim \gamma$. Thus, integrating next in the $x$-direction,
$$
  \abs{E} \lesssim \Lmin^{12} \frac{\Lmax^{12}}{\gamma} \abs{ \left\{ \xi^2 \colon \abs{\xi^2-b} \lesssim N_0  \right\} }
  \lesssim \Lmin^{12} \frac{\Lmax^{12}}{\gamma} N_0
$$
as desired, where $b$ denotes the second coordinate of $\xi_1'$.

\section{Proof of the anisotropic estimate}\label{K}

Here we prove Theorem \ref{Z:Thm1}. By tiling and duality it suffices to prove, given any $\delta >0$ and intervals $I_1,I_2 \subset \R$ with $\abs{I_1} = \abs{I_2} = \delta$, that
\begin{equation}\label{K:3:1}
  \iint \overline{u_0}\, u_1^{I_1} u_2^{I_2} \, dt \, dx
  \lesssim \left[ \frac{\delta(\Nmin^{12})^{1/2} (L_1L_2)^{3/4}}{\alpha} \right]^{1/2}
  \bignorm{u_0}
  \bignorm{u_1^{I_1}}
  \bignorm{u_2^{I_2}},
\end{equation}
where $u_j^{I_j} = \Proj_{\xi_j \cdot \omega \in I_j} u_j$ and we may assume $\widehat{u_j} \ge 0$ for $j=0,1,2$. Splitting the support of $\widehat{u_1}$ into two symmetric parts, and replacing $\alpha$ by $2\alpha$, we may assume
$$
  \supp \widehat{u_1} \subset A_0 \equiv \left\{ (\tau,\xi) \colon
  \theta(\pm_1\xi,\omega) \le \frac{\pi}{2} - 2\alpha \right\}.
$$
Next, split the support of $\widehat{u_2}$ into three parts, by intersecting with
\begin{align*}
  A_1 &= \left\{ (\tau,\xi) \colon
  \theta(\xi,\omega^\perp) \le \alpha \right\},
  \\
  A_2 &= \left\{ (\tau,\xi) \colon
  \theta(\pm_2\xi,\omega) \le \frac{\pi}{2} - \alpha \right\},
  \\
  A_3 &= \left\{ (\tau,\xi) \colon
  \theta(\pm_2\xi,-\omega) \le \frac{\pi}{2} - \alpha \right\}.
\end{align*}
Correspondingly we split the proof into three cases.

\subsection{The case $\supp \widehat{u_2} \subset A_1$}\label{K:8}

By Lemma \ref{B:Lemma1} we reduce to proving the volume bound $\abs{E} \lesssim \alpha^{-1}\delta(\Nmin^{12})^{1/2} (L_1L_2)^{3/4}$ uniformly in $(\tau_0,\xi_0) \in \R^{1+2}$, where
$$
  E = \left\{ (\tau,\xi) : \xi \cdot \omega \in I_1 \right\} \cap \dot K^{\pm_1}_{N_1,L_1} 
  \cap A_0 \cap \left[ (\tau_0,\xi_0) - A_1 \cap \dot K^{\pm_2}_{N_2,L_2} \right].
$$
This is of the form \eqref{K:12} with
$$
  R \subset \Bigl\{ \xi \colon  \; \abs{\xi} \sim N_1,
  \;
  \abs{\xi_0-\xi} \sim N_2,
  \;
  \xi \cdot \omega \in I_1,
  \;
  \theta(e_1,\omega^\perp) \ge 2\alpha,
  \;
  \theta(e_2,\omega^\perp) \le \alpha \Bigr\},
$$
where $e_1,e_2$ are defined as in \eqref{K:24}. So now \eqref{K:16} holds, with $f$ as in \eqref{K:20}. Choose coordinates $\xi=(\xi^1,\xi^2)$ so that $\omega = (1,0)$. Then for all $\xi \in R$, recalling that $\nabla f(\xi) = e_1 - e_2$, we have
$$
  \abs{\partial_2 f(\xi)}
  = \cos\theta(e_2,\omega^\perp) - \cos\theta(e_1,\omega^\perp)
  \ge \cos \alpha - \cos 2\alpha
  \sim \alpha^2,
$$
so integrating next in the $\xi^2$-direction we get
$$
  \abs{E} \lesssim \Lmin^{12} \frac{\Lmax^{12}}{\alpha^2} \abs{ \left\{ \xi^1 \colon \xi^1 \in I_1  \right\} }
  \lesssim \Lmin^{12} \frac{\Lmax^{12}}{\alpha^2} \, \delta,
$$
proving the desired estimate if $\alpha \gtrsim \gamma_0$, where
$$
  \gamma_0 = \sqrt{\frac{\Lmax^{12}}{\Nmin^{12}}}.
$$
If $\alpha \lesssim \gamma_0$, we use $\abs{E} \lesssim \delta \Nmin^{12} \Lmin^{12}$, which holds since $\abs{R} \lesssim \delta \Nmin^{12}$.

\subsection{The case $\supp \widehat{u_2} \subset A_2$}\label{K:38}

By Lemma \ref{A:Lemma4},
$$
  \text{l.h.s.}\eqref{K:3:1}
  \sim
  \sum_{\gamma} \sum_{\omega_1,\omega_2}
  \iint
  \overline{u_0} \, u_1^{I_1;\gamma,\omega_1}
  u_2^{I_2;\gamma,\omega_2} \, dt \, dx,
$$
where $0 < \gamma < 1$ is dyadic and $\omega_1,\omega_2 \in \Omega(\gamma)$ satisfy $3\gamma \le \theta(\omega_1,\omega_2) \le 12\gamma$. Depending on whether $\gamma \ge \gamma_0$ or $0 < \gamma < \gamma_0$ we then use, respectively,
\begin{align}\label{K:50}
  \bignorm{u_1^{I_1;\gamma,\omega_2}
  u_2^{I_2;\gamma,\omega_2}}
  &\lesssim \left( \frac{\delta L_1 L_2}{\gamma\alpha} \right)^{1/2}
  \bignorm{u_1^{I_1;\gamma,\omega_1}}
  \bignorm{u_2^{I_2;\gamma,\omega_2}},
  \\
  \label{K:52}
  \bignorm{u_1^{I_1;\gamma,\omega_2}
  u_2^{I_2;\gamma,\omega_2}}
  &\lesssim \left( \frac{\delta \Nmin^{12}\gamma \Lmin^{12}}{\alpha} \right)^{1/2}
  \bignorm{u_1^{I_1;\gamma,\omega_1}}
  \bignorm{u_2^{I_2;\gamma,\omega_2}},
\end{align}
and summing $\omega_1,\omega_2$ as in \eqref{A:208}, we then obtain \eqref{K:3:1}. This concludes the case $\supp\widehat{u_2} \subset A_2$, up to the claimed estimates \eqref{K:50} and \eqref{K:52}, which we now prove.

We reduce \eqref{K:50} to $\abs{E} \lesssim \delta L_1 L_2/(\gamma\alpha)$ for $E$ as in \eqref{K:12}, with
$$
  R = \Bigl\{ \xi \colon \xi \cdot \omega \in I_1,
  \;
  \theta(e_1,e_2) \ge \gamma,
  \\
  \theta(e_1,\omega) \le \frac{\pi}{2} - \alpha,
  \;
  \theta(e_2,\omega) \le \frac{\pi}{2} - \alpha \Bigr\}.
$$
Choosing coordinates so that $\omega = (1,0)$ we then have, for all $\xi \in R$,
$$
  \abs{\partial_2 f(\xi)}
  = \abs{ \cos\theta(e_1,\omega^\perp) - \cos\theta(e_2,\omega^\perp)}
  \ge \cos \alpha - \cos (\alpha+\gamma)
  \sim \gamma\alpha.
$$
Integrating next in the $\xi^2$-direction, we therefore get
$$
  \abs{E} \lesssim \Lmin^{12} \frac{\Lmax^{12}}{\gamma\alpha} \abs{ \left\{ \xi^1 \colon \xi^1 \in I_1  \right\} }
  \lesssim \frac{L_1L_2}{\gamma\alpha} \delta
$$
as desired.

For \eqref{K:52} we need, assuming $N_1 \le N_2$ by symmetry, $\abs{E} \lesssim \delta N_1\gamma \Lmin^{12}/\alpha$ with
$$
  R = \left\{ \xi \colon  \; \abs{\xi} \sim N_1,
  \;
  \xi \cdot \omega \in I_1,
  \;
  \theta(e_1,\omega_1) \le \gamma,
  \; 
  \theta(e_1,\omega^\perp) \ge \alpha
  \right\}.
$$
Since $\abs{E} \lesssim \Lmin^{12} \abs{R}$, it suffices to show $ \abs{R} \lesssim \delta N_1\gamma/\alpha$, but this is easy; we omit the details.

\subsection{The case $\supp \widehat{u_2} \subset A_3$}\label{K:298}

Assume $\abs{\xi_1} \le \abs{\xi_2}$ by symmetry, hence $N_1 \lesssim N_2$. In fact we may assume $N_1 \ll N_2$, since if $N_1 \sim N_2$ we can apply the bilinear trick
$$
  \bignorm{u_1^{I_1} u_2^{I_2}}
  \le
  \bignorm{u_1^{I_1}}_{L^4} \bignorm{u_2^{I_2}}_{L^4}
  =
  \bignorm{u_1^{I_1} u_1^{I_1}}^{1/2} \bignorm{u_2^{I_2} u_2^{I_2}}^{1/2}
$$
to effectively reduce to the case $\supp \widehat{u_2} \subset A_2$.

We need $\abs{E} \lesssim \alpha^{-1}\delta(\Nmin^{12})^{1/2} (L_1L_2)^{3/4}$ with
$$
  R = \Bigl\{ \xi \colon  \; \abs{\xi} \sim N_1,\;
  \abs{\xi_0-\xi} \sim N_2,
  \;
  \xi \cdot \omega \in I_1, \;
  \theta(e_1,e_2) \ge 2\alpha \Bigr\}.
$$
Since $N_1 \ll N_2$, we must have $\abs{\xi_0} \sim N_2$ (otherwise $R$ would be empty). Choose coordinates so that $\omega = (1,0)$. Then $\partial_1 f(\xi) \gtrsim \alpha$ for $\xi \in R$, so using \eqref{K:16} we get
\begin{equation}\label{K:324}
  \abs{E} \lesssim \Lmin^{12} \min\left(\delta,\frac{\Lmax^{12}}{\alpha}\right)
  \abs{ \boldsymbol{\pi}_{\omega^\perp}(S \cap \{ \xi \colon \xi^1 \in I \} )},
\end{equation}
where
$$
  S = \left\{ \xi \colon \; \abs{\xi} \lesssim N_1, \; \theta_{12} \ge 2\alpha, \; f(\xi) = \tau_0 + O\bigl( \Lmax^{12} \bigr) \right\}.
$$
Here we write $\theta_{12} = \theta(e_1,e_2)$, $\boldsymbol{\pi}_{\omega^\perp}$ denotes the orthogonal projection onto $\omega^\perp$, and $f$ is given by \eqref{K:20}, so in particular $\nabla f(\xi) = e_1 - e_2$.

On the other hand we obviously have $\abs{E} \lesssim \Lmin^{12} \abs{R} \lesssim \Lmin^{12} \delta N_1$, and this implies the desired estimate if $\alpha \lesssim (\Lmax^{12}/N_1)^{1/2}$, hence we may assume
$$
  \Lmax^{12} \ll N_1\alpha^2,
$$
and since $\theta_{12} \ge 2\alpha$ it then follows by Lemma \ref{I:Lemma1} that
\begin{equation}\label{K:328}
  \Lmax^{12}
  \ll
  \bigabs{\abs{\tau_0} - \abs{\xi_0}} \sim
  N_1\theta_{12}^2.
\end{equation}
Note that
$$
  S \subset \left\{ \xi \colon \abs{\xi} \sim N_1, \; \theta_{12} \ge 2\alpha \right\} \cap 
  \bigcup_{\tau_1 < \tau < \tau_2} S(\tau),
$$
where $\tau_1 < \tau_0 < \tau_2$ with $\tau_2-\tau_1 \sim \Lmax^{12}$ and
$$
  S(\tau) = \left\{ \xi \colon  \abs{\xi} \le \abs{\xi_0-\xi}, \; f(\xi) = \tau \right\}
$$
is the left half of an ellipse [resp.~the left branch of a hyperbola] if $\pm_1=\pm_2$ [resp.~$\pm_1\neq\pm_2$] with foci at $0$ and $\xi_0$ and center at $\xi_0/2$. Let us now switch to coordinates $\xi=(\xi^1,\xi^2)$ so that $\xi_0 = (\abs{\xi_0},0)$. Note that the condition $\abs{\xi} \le \abs{\xi_0-\xi}$ means that $\xi$ lies in the halfplane to the left of the vertical line $\xi^1 = \abs{\xi_0}/2$ through the midpoint of the foci. The family of curves $S(\tau)$ foliates this halfplane.

We claim (this is proved below) that if we start at a point $\xi$ in the halfplane $\abs{\xi} \le \abs{\xi_0-\xi}$ and follow the integral curve of $\nabla f$ out from this point, the angle $\theta_{12}$ will only increase as long as the integral curve stays in the halfplane. Note also that $\abs{\nabla f} = \abs{e_1-e_2} \sim \theta_{12}$. Therefore, if we start at a point $\xi \in S(\tau)$ with $\tau \in (\tau_1,\tau_2)$, and where $\theta_{12} \ge 2\alpha$, then along the integral curve of $\nabla f$ starting from this point we will have $\abs{\nabla f} \gtrsim \alpha$ until we hit the outer boundary $S(\tau_2)$. Therefore we will hit the outer boundary after moving a distance $d \lesssim \Lmax^{12}/\alpha \ll N_1$. From this we conclude that $S$ is contained in an $\varepsilon$-neighborhood, where $\varepsilon \sim \Lmax^{12}/\alpha$, of the convex curve
$$
  C_0 = \left\{ \xi \colon \abs{\xi} \sim N_1, \; \theta_{12} \ge 2\alpha \right\} \cap S(\tau_2).
$$
Moreover, the neighborhood is tubular since, by the second inequality in \eqref{Curvature} below, the minimal radius of curvature $R_*$ of $C_0$ satisfies $R_* \gtrsim N_1 \gg \varepsilon$, where the last inequality holds since $\Lmax^{12} \ll N_1\alpha^2$.

Next we claim that the curvature $\kappa$ of $C_0$ satisfies
\begin{equation}\label{Curvature}
  \frac{\alpha}{N_1} \lesssim \kappa \lesssim \frac{1}{N_1}.
\end{equation}
Granting this claim for the moment, we conclude that the maximal radius of curvature $R^*$ of $C_0$ satisfies $R^* \lesssim N_1/\alpha$. Then it follows that the intersection between a tubular $\varepsilon$-neighborhood of the convex curve $C_0$ and the $\delta$-thickened line $\{ \xi \colon \xi^1 \in I \}$ has a diameter no larger than $\sqrt{R^* \max(\delta,\varepsilon)}$, so from \eqref{K:324} we get
$$
  \abs{E} \lesssim \Lmin^{12} \min\left(\delta,\frac{\Lmax^{12}}{\alpha}\right)
  \sqrt{\frac{N_1}{\alpha} \max\left(\delta,\frac{\Lmax^{12}}{\alpha}\right)}
  \le \frac{\delta(\Nmin^{12})^{1/2} (L_1L_2)^{3/4}}{\alpha},
$$
as desired.

It remains to prove the two claims made above. Without loss of generality we assume $\pm_1 = +$, so that $f(\xi) = \abs{\xi} \pm \abs{\xi_0-\xi}$, where the plus sign gives an ellipse and the minus sign a hyperbola.

\subsubsection{Proof of \eqref{Curvature}} Assuming that $S$ is nonempty, then in view of \eqref{K:328} we have that, for any $\tau \in (\tau_1,\tau_2)$, the major and minor semiaxes $a$ and $b$ of $S(\tau)$ satisfy
$$
  2a = \abs{\tau} = \bigabs{\abs{\xi} \pm \abs{\xi_0-\xi}} \sim N_2,
  \qquad
  2b = \bigabs{\tau^2 - \abs{\xi_0}^2}^{1/2} \sim (N_1N_2)^{1/2} \theta_{12}.
$$
Parametrize $S(\tau)$ by $\xi^1=r\cos\theta$, $\xi^2=r\sin\theta$, where $\theta=\theta(\xi,\xi_0)$ and
$$
  r= \frac{b^2}{a \mp c\cos\theta},
$$
where $c=\abs{\xi_0}/2 \sim N_2$. Then (here the dots denote a $\theta$-derivatives)
$$
  \kappa = \frac{\abs{r^2 + 2\dot r^2 - r\ddot r}}{(r^2+\dot r^2)^{3/2}}.
$$
Since $r = \abs{\xi} \sim N_1$ and $b^2 \sim N_1N_2\theta_{12}^2$, we conclude that $a \mp c\cos\theta \sim N_2\theta_{12}^2$. We then calculate
$$
  \dot r = \mp \frac{b^2}{(a \mp c\cos\theta)^2} c\sin\theta
$$
and
$$
  \ddot r = \mp \frac{b^2}{(a \mp c\cos\theta)^2} c\cos\theta
  + \frac{b^2}{(a \mp c\cos\theta)^3} 2(c\sin\theta)^2,
$$
hence
$$
  r^2 + 2\dot r^2 - r\ddot r
  =
  \frac{a b^4}{(a \mp c\cos\theta)^3}
  \sim
  \frac{N_1^2}{\theta_{12}^2}.
$$
Moreover,
$$
  \abs{\dot r}
  \sim
  \frac{b^2}{(a \mp c\cos\theta)^2} N_2\sin\theta_{12}
  \lesssim
  \frac{N_1N_2\theta_{12}^2}{(N_2\theta_{12}^2)^2}N_2\theta_{12}
  \sim \frac{N_1}{\theta_{12}},
$$
where we used the fact that $\abs{\xi_0} \sin \theta = \abs{\xi_2} \sin\theta_{12}$. Plugging these facts into the formula for $\kappa$ above, and recalling that $\theta_{12} \ge \alpha$, we get the left inequality in\eqref{Curvature}. To prove the right inequality we split into the cases $\theta_{12} \ll 1$ and $\theta_{12} \sim 1$. If $\theta_{12} \ll 1$ we have (since $\sin\theta_{12} \sim \theta_{12}$) $\abs{\dot r} \sim N_1/\theta_{12} \gg r \sim N_1$, hence $\kappa \sim \theta_{12}/N_1$. If $\theta_{12} \sim 1$, on the other hand, then we estimate $\kappa \lesssim N_1^2/r^3 \sim 1/N_1$.

\subsubsection{Proof of the claim about integral curves}

The claim is that if we start at a point $\xi$ in the halfplane $\abs{\xi} \le \abs{\xi_0-\xi}$ and follow the integral curve of $\nabla f$ out from this point, the angle $\theta_{12}$ will only increase as long as the integral curve stays in the halfplane. But $\abs{\nabla f}^2 = \abs{e_1-e_2}^2 = 2(1-\cos\theta_{12})$, so what we have to prove is that $\abs{\nabla f}^2$ increases along the integral curves of $\nabla f$, as long as $\abs{\xi} \le \abs{\xi_0-\xi}$. Thus, it is enough to check that $\nabla\left( \abs{\nabla f}^2 \right) \cdot \nabla f > 0$ for $\abs{\xi} < \abs{\xi_0-\xi}$. But a direct calculation reveals that
$$
  \nabla\left( \abs{\nabla f}^2 \right) \cdot \nabla f
  =
  (\xi^2)^2
  \frac{\abs{\xi_0-\xi} \pm \abs{\xi}}{\abs{\xi_0-\xi}^3\abs{\xi}^3},
$$
so the desired positivity is clear.

\section{Proof of the null form estimate}\label{E}

Here we prove Theorem \ref{N:Thm2}. By duality, write the estimate as
\begin{equation}\label{E:24}
  \iint \overline{u_0}\, \nullform(\Proj_{\R \times T_r(\omega)} u_1,u_2) \, dt \, dx
  \lesssim \left( r L_1 L_2 \right)^{1/2}
  \norm{u_0}
  \norm{u_1}
  \norm{u_2},
\end{equation}
where $\widehat{u_j} \subset \dot K^{\pm_j}_{N_j,L_j}$ for $j=1,2$, and without loss of generality $\widehat{u_j} \ge 0$ for $j=0,1,2$. By Lemma \ref{A:Lemma4},
\begin{equation}\label{E:18}
  \text{l.h.s.}\eqref{E:24}
  \sim
  \sum_{\gamma} \sum_{\omega_1,\omega_2}
  \gamma
  \iint
  \overline{u_0}
  \left( \Proj_{\R \times T_r(\omega)} u_1^{\gamma,\omega_1} \right)
  u_2^{\gamma,\omega_2} \, dt \, dx,
\end{equation}
where $0 < \gamma < 1$ is dyadic and $\omega_1,\omega_2 \in \Omega(\gamma)$ with $3\gamma \le \theta(\omega_1,\omega_2) \le 12\gamma$. We claim that the following hold (these are proved below):
\begin{align}
  \label{E:20}
  \norm{\Proj_{\R \times T_r(\omega)} u_1^{\gamma,\omega_1}
  \cdot u_2^{\gamma,\omega_2}}
  &\lesssim \left( \frac{r L_1 L_2}{\gamma^2} \right)^{1/2}
  \norm{u_1^{\gamma,\omega_1}}
  \norm{u_2^{\gamma,\omega_2}},
  \\
  \label{E:21}
  \norm{\Proj_{\R \times T_r(\omega)} u_1^{\gamma,\omega_1}
  \cdot u_2^{\gamma,\omega_2}}
  &\lesssim \left( \frac{\Nmin^{12} L_1 L_2}{\gamma}\right)^{1/2}
  \norm{u_1^{\gamma,\omega_1}}
  \norm{u_2^{\gamma,\omega_2}}
  \\
  \label{E:22}
  \norm{\Proj_{\R \times T_r(\omega)} u_1^{\gamma,\omega_1}
  \cdot u_2^{\gamma,\omega_2}}
  &\lesssim \left( r \Nmin^{12} \Lmin^{12} \right)^{1/2}
  \norm{u_1^{\gamma,\omega_1}}
  \norm{u_2^{\gamma,\omega_2}}.
\end{align}

In the case
$$
  0 < \gamma \lesssim \gamma_0 \equiv \max \left( 
  \left(\frac{\Lmax^{12}}{\Nmin^{12}}\right)^{1/2},
  \frac{r}{\Nmin^{12}} \right)
$$
we apply \eqref{E:21} if $r > (\Nmin^{12}\Lmax^{12})^{1/2}$, \eqref{E:22} otherwise, and we sum $\omega_1,\omega_2$ as in~\eqref{A:208}. This gives the desired estimate since $\sum_{0 < \gamma \lesssim \gamma_0} \gamma^p \sim \gamma_0^p$ for $p > 0$.

It then remains to consider the case
$$
  \gamma_0 \ll \gamma < 1.
$$
Now we use \eqref{E:20}, but to avoid a logarithmic loss when summing $\gamma$, we need to exploit some orthogonality in the bilinear interaction. To begin with, observe that since $\xi_1 \in T_r(\omega)$ and $\abs{\xi_1} \sim N_1$, we may assume (replacing $\omega$ by $-\omega$ if necessary) that
$\theta(\pm_1\xi_1,\omega) \lesssim r/N_1 \ll \gamma$, where the last inequality is due to $\gamma_0 \ll \gamma$. Moreover, $\theta(\pm_1\xi_1,\omega_1) \le \gamma$, hence $\theta(\omega_1,\omega) \le 3\gamma/2$, implying that $\omega_1 \in \Omega(\gamma)$ is essentially uniquely determined, hence so is $\omega_2$.

Since $\gamma_0 \ll \gamma$, we have $\Nmin^{12}\gamma^2 \gg \Lmax^{12}$, and $3\gamma \le \theta(\omega_1,\omega_2) \le 12\gamma$ implies $\theta_{12} \sim \gamma$, hence Lemma \ref{I:Lemma1} gives
$$
  \abs{\hypwt_0} \equiv \bigabs{\abs{\tau_0}-\abs{\xi_0}} \sim
  \left\{
  \begin{alignedat}{2}
  &\Nmin^{12}\gamma^2& \quad &\text{if $\pm_1=\pm_2$},
  \\
  &\frac{N_1N_2\gamma^2}{\abs{\xi_0}}& \quad &\text{if $\pm_1\neq\pm_2$}.
  \end{alignedat}
  \right.
$$
In the case $\pm_1=\pm_2$ we can therefore estimate the sum in \eqref{E:18} by an absolute constant times
$$
  \sum_{\gamma} \sum_{\omega_1,\omega_2}
  \left(rL_1L_2\right)^{1/2}
  \bignorm{\Proj_{\abs{\hypwt_0} \sim \Nmin^{12}\gamma^2} u_0}
  \bignorm{u_1^{\gamma,\omega_1}}
  \bignorm{u_2^{\gamma,\omega_2}}
  \le
  \left(rL_1L_2\right)^{1/2} A B,
$$
where
\begin{align*}
  A^2 &=
  \sum_{\gamma} \sum_{\omega_1,\omega_2}
  \bignorm{\Proj_{\abs{\hypwt_0} \sim \Nmin^{12}\gamma^2} u_0}^2
  \sim
  \sum_{\gamma} \bignorm{\Proj_{\abs{\hypwt_0} \sim \Nmin^{12}\gamma^2} u_0}^2
  \sim \norm{u_0}^2,
  \\
  B^2 &= \sum_{\gamma} \sum_{\omega_1,\omega_2}
  \bignorm{u_1^{\gamma,\omega_1}}^2
  \bignorm{u_2^{\gamma,\omega_2}}^2
  \sim \bignorm{u_1}^2 \bignorm{u_2}^2.
\end{align*}
Here we used the observation that $\omega_1,\omega_2$ are essentially uniquely determined to get the estimate for $A^2$, and we used Lemma \ref{A:Lemma4} to get the estimate for $B^2$. The case $\pm_1 \neq \pm_2$ works out the same way except that we use $\abs{\hypwt_0} \sim \frac{N_1 N_2\gamma^2}{\abs{\xi_0}}$. 

\subsection{Proof of \eqref{E:20}--\eqref{E:22}}\label{E:38:2} We need to bound $\abs{E}$ by an absolute constant times, respectively, $rL_1L_2/\gamma^2$, $\Nmin^{12} L_1L_2/\gamma$ and $r\Nmin^{12}\Lmin^{12}$, where $E$ satisfies \eqref{K:12} with
$$
  R = \bigl\{ \xi \in T_r(\omega) \colon  \abs{\xi} \sim N_1, \;
  \abs{\xi_0-\xi} \sim N_2, \; \theta(e_1,\omega_1) \le \gamma, \;
  \; \theta(e_2,\omega_2) \le \gamma \bigr\}
$$
and $e_1,e_2$ are defined as in \eqref{K:24}. Then \eqref{K:16} holds with $f$ as in \eqref{K:20}. Clearly, $\abs{R} \lesssim r\Nmin^{12}$, proving \eqref{E:22}. The estimate $\abs{E} \lesssim \Nmin^{12} L_1L_2/\gamma$ follows as in the proof of \eqref{C:200}, and this proves \eqref{E:21}. So it remains to prove \eqref{E:20}. Let $\xi \in R$. We may assume $r \ll N_1\gamma$ (otherwise \eqref{E:21} is better), hence (replacing $\omega$ by $-\omega$ if necessary) $\theta(e_1,\omega) \lesssim r/N_1 \ll \gamma$, and since $\theta(e_1,e_2) \ge \gamma$ it follows that $\theta(e_2,\omega) \ge \frac12 \gamma$. Thus,
$$
  \nabla f(\xi) \cdot \omega = (e_1-e_2)\cdot\omega
  = \cos \theta(e_1,\omega) - \cos \theta(e_2,\omega)
  \ge \cos\frac14\gamma-\cos\frac12 \gamma \sim \gamma^2,
$$
so from \eqref{K:16} we conclude that
$$
  \abs{E} \lesssim \Lmin^{12} \frac{\Lmax^{12}}{\gamma^2}
  \bigabs{\boldsymbol{\pi}_{\omega^\perp}(T_r(\omega))}
  \lesssim \Lmin^{12} \frac{\Lmax^{12}}{\gamma^2}
  r,
$$
where $\boldsymbol{\pi}_{\omega^\perp}$ is the projection onto $\omega_1^\perp$. This proves \eqref{E:20}.

\section{Proof of the concentration/nonconcentration estimate}\label{J}

Here we prove Theorem \ref{N:Thm4}. By duality and Lemma \ref{A:Lemma4} we reduce to
\begin{multline}\label{J:190}
  \sum_{\gamma} \sum_{\omega_1,\omega_2}
  \gamma
  \iint
  \overline{u_0}
  \, \Proj_{\xi_0 \cdot \omega \in I_0}
  \left( \Proj_{\R \times T_r(\omega)} u_1^{\gamma,\omega_1} \right)
  u_2^{\gamma,\omega_2} \, dt \, dx
  \\
  \lesssim
  \left(rL_1L_2\right)^{1/2}
  \norm{u_0}
  \left(\sup_{I_1} \norm{\Proj_{\xi_1 \cdot \omega \in I_1} u_1}\right)
  \norm{u_2}
\end{multline}
where $0 < \gamma \ll 1$ is dyadic  and $\omega_1,\omega_2 \in \Omega(\gamma)$ satisfy $3\gamma \le \theta(\omega_1,\omega_2) \le 12\gamma$. Since $L_1,L_2$ appear to the power $1/2$ in the right hand side, we can assume that they are arbitrarily small, by dividing the thickened cones into thinner cones and summing the resulting estimates using Cauchy-Schwarz (see Remark 4.1 in \cite{Selberg:2008c} for more details). In particular, we can assume $L_2 \ll N_2$, which will be needed at a certain point later on.

Define $\gamma_0$ as in the previous section. For $0 < \gamma \lesssim \gamma_0$, we argue as in the proof of Theorem \ref{N:Thm2}, but instead of \eqref{E:21} and \eqref{E:22} we use the estimates (proved below)
\begin{align}
  \label{J:200}
  \norm{\Proj_{\xi_0 \cdot \omega \in I_0} \left(
  \Proj_{\R \times T_r(\omega)} u_1^{\gamma,\omega_1} \cdot u_2^{\gamma,\omega_2} \right)}
  &\lesssim \left( r \abs{I_0} \Lmin^{12} \right)^{1/2} 
  \norm{u_1^{\gamma,\omega_1}}
  \norm{u_2^{\gamma,\omega_2}},
  \\
  \label{J:202}
  \norm{\Proj_{\xi_0 \cdot \omega \in I_0} \left(
  \Proj_{\R \times T_r(\omega)} u_1^{\gamma,\omega_1} \cdot u_2^{\gamma,\omega_2} \right)}
  &\lesssim \left( \frac{\abs{I_0} L_1 L_2}{\gamma} \right)^{1/2} 
  \norm{u_1^{\gamma,\omega_1}}
  \norm{u_2^{\gamma,\omega_2}}.
\end{align}
If we also tile by the condition $\xi_0 \cdot \omega \in I_0$, then we see that the part of \eqref{J:190} corresponding to $0 < \gamma \lesssim \gamma_0$ is dominated by
\begin{equation}\label{J:210}
  \left(\frac{\abs{I_0}}{\Nmin^{12}}\right)^{1/2}
  \left(rL_1L_2\right)^{1/2}
  \sum_{I_1,I_2}
  \bignorm{u_0}
  \bignorm{u_1^{I_1}}
  \bignorm{u_2^{I_2}}
  \qquad \left( u_j^{I_j} = \Proj_{\xi_j \cdot \omega \in I_j} u_j \right),
\end{equation}
where $I_1, I_2$ belong to the almost disjoint cover of $\R$ by translates of $I_0$, and the sum is restricted by the condition $(I_1 + I_2) \cap I_0 \neq \emptyset$, hence the sum is over a set of cardinality comparable to $\Nmin^{12}/\abs{I_0}$, and each $I_2$ can interact with at most three different $I_1$'s. Thus, sup'ing over $I_1$ and summing $I_2$ using the Cauchy-Schwarz inequality, we get the bound in the right hand side of \eqref{J:190}.

It remains to consider $\gamma_0 \ll \gamma \ll 1$. Then we use the estimate (proved below)
\begin{equation}\label{J:220}
  \norm{\Proj_{\xi_0 \cdot \omega \in I_0} \left(
  \Proj_{\R \times T_r(\omega)} u_1^{\gamma,\omega_1} \cdot u_2^{\gamma,\omega_2} \right)}
  \lesssim \left( \frac{rL_1L_2}{\gamma^2} \right)^{1/2} 
  \left( \sup_{I_1} \bignorm{u_1^{I_1}} \right)
  \bignorm{u_2^{\gamma,\omega_2}}.
\end{equation}
To avoid a logarithmic loss when summing $\gamma$, we repeat the argument from the end of section \ref{E}, with the difference that now
$$
  B^2 = \left(\sup_{I_1} \bignorm{u_1^{I_1}}\right)^2
  \sum_{\gamma} \sum_{\omega_1,\omega_2}
  \bignorm{u_2^{I_2;\gamma,\omega_2}}^2.
$$
Since $r \ll \Nmin^{12}\gamma$, we may assume $\theta(\pm_1\xi_1,\omega) \lesssim r/N_1 \ll \gamma$, so $\omega_1, \omega_2 \in \Omega(\gamma)$ are essentially uniquely determined, and $\theta(\omega_2,\omega) \ge (3/2)\gamma$, hence $\theta(\pm_2\xi_2,\omega) \sim \gamma$. Thus,
$$
  \sum_{\gamma} \sum_{\omega_1,\omega_2}
  \bignorm{u_2^{I_2;\gamma,\omega_2}}^2
  \sim
  \sum_{\gamma}
  \bignorm{\Proj_{\theta(\pm_2\xi_2,\omega) \sim \gamma} u_2^{I_2}}^2
  \sim \bignorm{u_2^{I_2}}^2,
$$
so the argument at the end of section \ref{E} goes through.

It remains to prove the claimed estimates \eqref{J:200}, \eqref{J:202} and \eqref{J:220}.

\subsection{Proof of \eqref{J:200} and \eqref{J:202}}

By Lemma \ref{B:Lemma2} it is enough to prove $\abs{E} \lesssim r\abs{I_0}\Lmin^{12}$ and $\abs{E} \lesssim \abs{I_0}L_1L_2/\gamma$, where $E$ satisfies \eqref{K:12} with
$$
  R = \bigl\{ \xi \in T_r(\omega) \colon \xi \cdot \omega \in I_1, \; \abs{\xi} \sim N_1, \; \theta(e_1,\omega_1) \le \gamma, \;
  \; \theta(e_2,\omega_2) \le \gamma \bigr\}
$$
and $I_1$ is a translate of $I_0$. Trivially, $\abs{R} \lesssim r\abs{I_0}$, proving \eqref{J:200}. To prove \eqref{J:202}, choose coordinates $\xi=(\xi^1,\xi^2)$ so that $\omega_1 = (1,0)$. Since $\gamma \ll 1$, it is clear that $e_1-e_2$ is roughly perpendicular to $\omega_1$, hence $\abs{\partial_2 f} = \abs{(e_1-e_2) \cdot (0,1)} \sim \abs{e_1-e_2} \sim \gamma$. So using \eqref{K:16} and integrating next in the $\xi^2$-direction, we get
$$
  \abs{E} \lesssim \Lmin^{12} \frac{\Lmax^{12}}{\gamma} \abs{\{\xi^1 : \xi^1 \in I_1\} } = \frac{\abs{I_0}L_1L_2}{\gamma},
$$
proving \eqref{J:202}.

\subsection{Proof of \eqref{J:220}}

To simplify the notation we shall write $u_j$ instead of $u_j^{\gamma,\omega_j}$. We want to prove that
\begin{equation}\label{J:250}
  \norm{\Proj_{\xi_0 \cdot \omega \in I_0} \left(u_1u_2\right)}
  \le C
  \left( \sup_{I_1} \norm{u_1^{I_1}} \right)
  \norm{u_2}
\end{equation} 
holds with $C^2 \sim r L_1 L_2/\gamma^2$, where the supremum is over all translates $I_1$ of $I_0$. Since $r \ll \Nmin^{12}\gamma$, we may assume (replacing $\omega$ by $-\omega$ if necessary) that $\Gamma_\phi(\omega) \subset \Gamma_\gamma(\omega_1)$, where
$$
  \phi \sim \frac{r}{N_1} \ll \gamma.
$$
Thus, $\supp\widehat{u_1} \subset \bigl[ \R \times T_{r}(\omega) \bigr] \cap K^{\pm_1}_{N_1,L_1,\phi,\omega}$ and $\supp\widehat{u_2} \subset K^{\pm_2}_{N_2,L_2,\gamma,\omega_2}$.

We shall make use of the following general facts:

\begin{lemma}\label{J:Lemma1} \emph{(\cite{Selberg:2008c}.)}
Given $\omega \in \mathbb S^1$, a compact interval $I_0$ and $S_1,S_2 \subset \R^{1+2}$, assume that for all $u_1,u_2 \in L^2(\R^{1+2})$ satisfying $\supp\widehat{u_j} \subset S_j$,
\begin{equation}\label{J:20}
  \norm{\Proj_{\xi_0 \cdot \omega \in I_0}(u_1u_2)}
  \le \left(A\abs{I_0}\right)^{1/2} \norm{u_1}\norm{u_2},
\end{equation}
where $A > 0$ is a constant. Assume further that there exist $d > 0$, $c \in \R$ and a compact interval $J$ such that
\begin{gather}
  \label{J:24}
  d \lesssim \abs{J}, \qquad
  S_2 \subset J \times \R^2,
  \\
  \label{J:28}
  S_2 \subset \left\{ (\tau_2,\xi_2) : -\tau_2 + \xi_2 \cdot \omega = c + O(d) \right\}.
\end{gather}
Then \eqref{J:250} holds with $C^2 \sim A \abs{J}$.
\end{lemma}

We also need:

\begin{lemma}\label{J:Lemma2} \emph{(\cite{Selberg:2008c}.)}
Suppose $\omega \in \mathbb S^1$, $I_0$ is a compact interval, $S_1,S_2 \subset \R^{1+2}$ and $S_1 \subset T_1$, where $T_1 \subset \R^{1+2}$ is an approximate tiling set with the doubling property. If
\begin{equation}\label{J:40}
  \norm{\Proj_{\xi_0 \cdot \omega \in I_0}(u_1\Proj_{T_2}u_2)}
  \le C_0
  \left( \sup_{I_1}\norm{u_1^{I_1}} \right)
  \norm{\Proj_{T_2}u_2}
\end{equation}
for all translates $T_2$ of $T_1$, all translates $I_1$ of $I_0$ and all $u_1,u_2 \in L^2(\R^{1+2})$ satisfying $\supp\widehat{u_j} \subset S_j$, then \eqref{J:250} also holds, with a constant $C$ depending on $C_0$ and the size of the overlap of the doubling cover by $T_1$.
\end{lemma}

Finally, recalling the notation $H_d(\omega) = \left\{ (\tau,\xi) \colon \abs{-\tau + \xi \cdot \omega} \lesssim d \right\}$ for a thickened null hyperplane, we note the elementary fact that
\begin{equation}\label{ConePlane}
  \dot K^{\pm}_{N,L,\gamma}(\omega) \subset H_{\max(L,N\gamma^2)}(\omega).
\end{equation}

We now apply the above lemmas to the specific situation that we are considering, namely
$$
  S_1 = \bigl[ \R \times T_{r}(\omega) \bigr] \cap K^{\pm_1}_{N_1,L_1,\phi,\omega},
  \qquad
  S_2 = K^{\pm_2}_{N_2,L_2,\gamma,\omega_2}.
$$
By \eqref{J:200} and \eqref{J:202}, we know that \eqref{J:20} holds with
$$
  A \sim \min\left(\frac{L_1L_2}{\gamma}, r \Lmin^{12} \right)
$$
By \eqref{ConePlane}, $S_1 \subset T_1$, where
$$
  T_1 = H_d(\omega) \cap \bigl( \R \times T_{r}(\omega) \bigr)
  \qquad
  d = \max(L_1,N_1\phi^2).
$$
Then $T_1$ is an approximate tiling set with the doubling property, so Lemma \ref{J:Lemma2} allows us to replace $S_2$ by $S_2 \cap T_2$ in Lemma \ref{J:Lemma1}, where $T_2 = (\tau_0,\xi_0) + T_1$ for some $(\tau_0,\xi_0) \in \R^{1+2}$. Clearly, \eqref{J:28} holds with $S_2$ replaced by $S_2 \cap T_2$, and with $c = -\tau_0+\xi_0 \cdot \omega$, so it only remains to prove the existence of an interval $J$ such that $S_2 \cap T_2 \subset J \times \R^2$ and
$$
  \abs{J} \sim \max\left(\frac{r}{\gamma}, \frac{\Lmax^{12}}{\gamma^2} \right) \gg d,
$$
where the last inequality holds by the definitions of $d$ and $\phi$ (recall that $N_1\phi^2 \sim r\phi \ll r\gamma$). The proof given in \cite{Selberg:2008c} for the existence of $J$ for the 3d case applies also in 2d, however, with only the obvious modifications. (This is where the assumption $L_2 \ll N_2$ is used.) The proof of Theorem \ref{N:Thm4} is now complete.

\section{Proof of the estimate for thickened circles}\label{KK}

Here we prove Lemma \ref{A:Lemma1}, i.e.
\begin{equation}\label{KK:18}
  \abs{\mathbb S^1_\delta(r) \cap [ \xi_0 + \mathbb S^1_\Delta(R)]}
  \lesssim \left[ \frac{rR\delta\Delta }{\abs{\xi_0}} \min(\delta,\Delta) \right]^{1/2}
\end{equation}
for $0 < \delta \ll r$, $0 < \Delta \ll R$ and $\xi_0 \in \R^2 \setminus \{0\}$.
Choose coordinates $\xi = (x,y)$ on $\R^2$ so that $\xi_0 = (\abs{\xi_0},0) \neq 0$. Then $\xi \in \mathbb S^1_\delta(r) \cap [ \xi_0 + \mathbb S^1_\Delta(R)]$ if and only if
$$
  (r-\delta)^2 < x^2 + y^2 < (r+\delta)^2
$$
and
$$
  (R-\Delta)^2 < (x-\abs{\xi_0})^2 + y^2 < (R+\Delta)^2.
$$
Subtracting these inequalities, we find that
\begin{multline*}
  \xi \in \mathbb S^1_\delta(r) \cap [ \xi_0 + \mathbb S^1_\Delta(R)]
  \implies x \in (a,b),
  \\
  \text{where} \quad
  \left\{
  \begin{aligned}
  a &= \frac{1}{2\abs{\xi_0}} \left[\abs{\xi_0}^2+r^2-R^2+\delta^2-\Delta^2-2(r\delta+R\Delta)\right],
  \\
  b &= \frac{1}{2\abs{\xi_0}} \left[\abs{\xi_0}^2+r^2-R^2+\delta^2-\Delta^2+2(r\delta+R\Delta)\right].
  \end{aligned}
\right.
\end{multline*}
But $b-a = 2(r\delta+R\Delta)/\abs{\xi_0}$, and by symmetry we may assume that $\delta \le \Delta$, so applying Lemma \ref{KK:Lemma2} below, we complete the proof as follows: First, if $\delta r \lesssim R\Delta$, then \eqref{KK:18} follows from \eqref{KK:20} in the lemma. Second, if $\delta r \gg R\Delta$, then $r \gg R(\Delta/\delta) \ge R$, so $\mathbb S^1_\delta(r) \cap [ \xi_0 + \mathbb S^1_\Delta(R) ]$ is empty unless $\abs{\xi_0} \sim r$. But the latter implies $b-a \sim \delta$, hence \eqref{KK:18} follows from \eqref{KK:22} applied with $(r,\delta)$ replaced by $(R,\Delta)$.

It then only remains to prove the following:

\begin{lemma}\label{KK:Lemma2} Let $a,b \in \R$ with $a < b$, and let $0 < \delta \ll r$. Then
\begin{align}
  \label{KK:20}
  &\abs{\mathbb S^1_\delta(r) \cap \left\{ \xi \in \R^2 \colon a < x < b \right\}}
  \lesssim \delta\sqrt{r(b-a)},
  \\
  \label{KK:22}
  &\abs{\mathbb S^1_\delta(r) \cap \left\{ \xi \in \R^2 \colon a < x < b \right\}}
  \lesssim (b-a)\sqrt{r\delta}
\end{align}
\end{lemma}

\begin{proof} Without loss of generality assume $0 \le a < b \le r+\delta$. We split into the cases (i) $b \le r-\delta$ and (ii) $r-\delta < b \le r+\delta$.

Set $y^* = \sqrt{(r+\delta)^2-x^2}$ and $y_* = \sqrt{(r-\delta)^2-x^2}$.

In case (i) we calculate the area as
\begin{multline*}
  2\int_a^b ( y^* - y_* ) d x
  \sim
  \int_a^b \frac{r\delta}{y^* + y_*} d x
  \sim
  \int_a^b \frac{r\delta}{\sqrt{r}\sqrt{r+\delta-x}} d x
  \\
  \sim
  \delta\sqrt{r}\left(\sqrt{r+\delta-a}-\sqrt{r+\delta-b}\right)
  \sim
  \delta\sqrt{r}\frac{b-a}{\sqrt{r+\delta-a}}.
\end{multline*}
But $r+\delta-a \ge b - a$, proving \eqref{KK:20}. On the other hand, $r+\delta-a \ge 2\delta$, since $a < b \le r-\delta$, so we have proved also \eqref{KK:22}.

In case (ii) we can set $a=r-\delta$, since the interval $a \le x \le r-\delta$ is covered by case (i). Therefore, the area is
$$
  2\int_a^b y^* d x 
  \sim
  \int_a^b \sqrt{r}\sqrt{r+\delta-x} d x
  \lesssim \int_a^b \sqrt{r}\sqrt{\delta} d x
  \le \sqrt{r\delta} (b-a)
  \lesssim \delta\sqrt{r(b-a)}.
$$
Here we used $r-\delta \le x \le r+\delta$ and $b-a \le 2\delta$.
\end{proof}

\bibliographystyle{amsalpha}
\bibliography{bibliography}

\providecommand{\bysame}{\leavevmode\hbox to3em{\hrulefill}\thinspace}
\providecommand{\MR}{\relax\ifhmode\unskip\space\fi MR }
\providecommand{\MRhref}[2]{%
  \href{http://www.ams.org/mathscinet-getitem?mr=#1}{#2}
}
\providecommand{\href}[2]{#2}
\begin{thebibliography}{Tao01}

\bibitem[FK00]{Foschi:2000}
D.~Foschi and S.~Klainerman, \emph{Homogeneous ${L}^{2}$ bilinear estimates for
  wave equations}, Ann. Scient. ENS $4^e$ serie \textbf{23} (2000), 211--274.

\bibitem[KM93]{Klainerman:1993}
S.~Klainerman and M.~Machedon, \emph{Space-time estimates for null forms and
  the local existence theorem}, Comm. Pure Appl. Math. \textbf{46} (1993),
  no.~9, 1221--1268.

\bibitem[KM96]{Klainerman:1996}
\bysame, \emph{Remark on {S}trichartz type inequalities}, Int. Math. Res. Not.
  (1996), no.~5, 201--220.

\bibitem[KS02]{Selberg:2002b}
S.~Klainerman and S.~Selberg, \emph{Bilinear estimates and applications to
  nonlinear wave equations}, Comm. Contemp. Math. \textbf{4} (2002), no.~2,
  223--295.

\bibitem[Sel08]{Selberg:2008c}
S.~Selberg, \emph{Anisotropic bilinear ${L}^2$ estimates related to the 3{D}
  wave equation}, Int. Math. Res. Not. (2008).

\bibitem[Str77]{Strichartz:1977}
R.~S. Strichartz, \emph{Restriction of {F}ourier transforms to quadratic
  surfaces and decay of solutions of wave equations}, Duke Math. J. \textbf{44}
  (1977), no.~3, 705--714.

\bibitem[Tao01]{Tao:2001}
T.~Tao, \emph{Multilinear weighted convolution of ${L}^{2}$ functions, and
  applications to nonlinear dispersive equations}, Amer. J. Math. \textbf{123}
  (2001), no.~5, 839--908.

\end{thebibliography}

\end{document}